\documentclass[12pt,a4paper,reqno]{amsart}
\usepackage{amsmath,amssymb,amsthm,amscd,mathrsfs} 
\usepackage[T1]{fontenc}
\usepackage[latin1]{inputenc}
\usepackage[english]{babel}
\usepackage[mathscr]{eucal}
\usepackage{graphicx}
\usepackage{float}
\usepackage{tikz}
\usetikzlibrary{positioning}
\usetikzlibrary{shapes.geometric}
\usetikzlibrary{fit}
\usetikzlibrary{patterns}
\usepackage{caption}
\usepackage{tikz-cd}
\usepackage{cases}
\usepackage{hyperref}
\usepackage{chngcntr}
\usepackage{enumitem} 
\usetikzlibrary{decorations.text}

\input{comment.sty}
\includecomment{FZ}
\includecomment{MM}
\counterwithin{equation}{section}
\renewcommand{\a}{\alpha}

\renewcommand{\th}{\theta}

\newcommand{\n}{\nu}

\newcommand{\ph}{\phi}
\def\ph{\phi}
\def\md#1{\ \mbox{\rm(mod }{#1})}

\def\npp#1{N_{\ph}^+(#1)}
\def\ol{\overline}

\def\ph{\phi}
\newcommand{\Q}{{\mathbb Q}}
\newcommand{\Z}{{\mathbb Z}}

\newcommand{\F}{{\mathbb F}}

\def\md#1{\ \mbox{\rm(mod }{#1})}

\def\npp#1{N_{\ph}^+(#1)}

\newtheorem{theorem}{Theorem}[section]
\newtheorem{proposition}[theorem]{Proposition}
\newtheorem{lemma}[theorem]{Lemma}
\newtheorem{corollary}[theorem]{Corollary}

\theoremstyle{definition}

\theoremstyle{remark}
\newtheorem{example}[theorem]{Example}
\newtheorem{remark}[theorem]{Remark}

\begin{document}
\title[]{ on  the  monogenity   of  pure number fields: \it{application to the  existence of canonical number systems} }
	\textcolor[rgb]{1.00,0.00,0.00}{}
	\author{  Hamid Ben Yakkou, Brahim Boudine and Pagdame Tiebekabe}\textcolor[rgb]{1.00,0.00,0.00}{}
	%	\author{  Hamid Ben Yakkou, Issam Aghzer  and Abdelkarim Boua }\textcolor[rgb]{1.00,0.00,0.00}{}
	\address{Polydisciplinary Faculty of Béni-Mellal, University Sultan Moulay Slimane, Béni-Mellal,  Morocco}\email{beyakouhamid@gmail.com }
	\address{Faculty of Sciences, Moulay Ismail University, Meknes, Morocco}\email{b.boudine@umi.ac.ma}
	\address{Department of Mathematics, Faculty of Sciences and Techniques, University of Kara, P. O. Box 43,
	Kara, Togo}\email{pagdame.tiebekabe@ucad.edu.sn}
	%\date{\today}
	\maketitle
	\vspace{0.3cm}
	\begin{abstract}

Let $m$ be a rational integer with $m \neq 0, \pm 1$, and consider the pure number field $K = \mathbb{Q}(\sqrt[n]{m})$ with $n \ge 3$. Most papers discussing the monogenity of pure number fields focus exclusively on the case where $m$ is square-free. For every integer $n \ge 4$, the monogenity of number fields of degree $n$ is not completely characterized. For example, the monogenity of the pure quartic field $\mathbb{Q}(\sqrt[4]{m})$ is not yet fully described, even when $m$ is square-free (see the recent 2024 paper \cite{Nyul} by Arn\'oczki and Nyul). In this paper, based on a classical theorem of Ore concerning prime ideal decomposition in number fields  \cite{MN92, O}, we study the monogenity of $K$ without assuming $m$ to be square-free. As an application, we present several examples related to canonical number systems (CNS).  In particular, we observe that our results extend some of those presented in \cite{BFC, BF, HNHCNS}.

	\end{abstract}
	\maketitle
	%\tableofcontents
	%\newpage

\section{Introduction}
Let $F(x) \in \mathbb{Z}[x]$ be a monic irreducible polynomial of degree $n$, and let $K = \mathbb{Q}(\alpha)$ be a number field generated by a root $\alpha$ of $F(x)$. Denote by $\mathbb{Z}_K$ the ring of integers of $K$.  A primitive element of $\mathbb{Z}_K$ is an element $\th$ verifying $K=\mathbb{Q}(\th)$. Let  $\widehat{\mathbb{Z}}_K$ be the set of all primitive elements of $\mathbb{Z}_K$. We say that $K$ is monogenic if there exists $\theta \in \widehat{\mathbb{Z}}_K$ such that $(1,\theta, \theta^2, \ldots,\theta^{n-1})$ is a basis of  the  $\mathbb{Z}$-module $\mathbb{Z}_K$.   In this case, we say that $\th$ is a generator of a power integral basis of $\Z_K$ (shortly, GPIB of $\Z_{K}$).

The monogenity of number fields is one of the most actively studied classical problems in algebraic number theory (cf. \cite{ANHN,EG,BGG6,G19,GG5,Gyoryrsurlespolynomes,Gyoryredecide,JonesRamanujan,PP,PethoZigler}). By Hensel's discovery in 1894, the monogenity of  $K$ is encoded by  the solvability of an index form equation \begin{equation}\label{IFE}
I(x_2,...,x_n) = \pm 1\,\, \mbox{in}\,\, x_2,..., x_n \in \mathbb{Z}
\end{equation} associated with an integral basis
$\{1, \omega_2,...,\omega_{n-1}\}$ of K (see \cite{EG,G19} and  \cite[2.2, Section 6, p. 64]{Na}). 

In \cite{Gyoryrsurlespolynomes,Gyoryredecide,GyorySeminarFrensh,Gyorybounds}, Gy\H{o}ry provided  general algorithms to decide whether $K$ is monogenic or not and to determine  all power integral bases in $\Z_K$, as well as effective bounds for the solutions of \eqref{IFE} and algorithms for solving index form equations. He  also studied in \cite{Gyoryrelative, Gyoryrdiscriminant} the monogenity of relative extensions. For comprehensive surveys on discriminant and index form theory, their applications to Diophantine equations, and monogenity of number fields, see   the   book  \cite{EG} by Evertse and  Gy\H{o}ry. 
 
For cubic number fields, see \cite{LN} by Llorente and Nart.   For quartic number fields, we refer to the works \cite{Bhargave} by   Alpöge,  Bhargava and  Shnidman,    \cite{Akhtari} by Akhtari, \cite{Nyul} by Arn\'oczki and Nyul,  \cite{DS} by Davis and Spearman, and \cite{GPP4} by Ga\'{a}l,  Peth\H{o} and Pohst. In \cite{GG5}, Ga\'{a}l and Gy\H{o}ry described an algorithm to solve index form equations in quintic number fields and made very extensive computations for power integral bases in totally real quintic number fields with Galois group $S_5$. In  \cite{BGG6}, Bilu, Ga\'{a}l and Gy\H{o}ry studied the monogenity of  sextic number fields. In \cite{PethoZigler}, Peth\H{o} and Ziegler gave an efficient criterion to decide whether the maximal order of a biquadratic field has a unit power integral basis or not. 
  For multiquadratic number fields, see \cite{PP} by Peth\H{o} and Pohst. In \cite{JonesRamanujan}, Jones studied the monogenity of sextic reciprocal dihedral polynomials.  For a survey on monogeneity, with a focus on efficient algorithms for various classes of number fields, see the books \cite{EG} by Evertse and Gy\H{o}ry and \cite{G19} by Ga\'{a}l. For an additional bibliography on monogeneity, see Narkiewicz's book \cite{Na}.

 Recently, the interest in the study of monogenity of pure number fields   $\Q(\sqrt[n]{m})$   increased clearly. Despite the lack of a standard classification, notable progress has been made in this direction.  There are several works that consider the case when $m$ is a square-free rational integer: for example   Gassert \cite{Gassert} investigated when $\Z_K=\Z[\a]$, Ahmad \emph{et al}.  \cite{ANHN, AN} explored the case $n=6$, Hameed and Nakahara \cite{HaNak8BullRom} studied the case $n=8$,  Ga\'al and Remete  studied  the cases $3 \le n \le 9$. %In \cite{ElFadil37}, El Fadil investigated the case $n=3^r\cdot 7^s$,
  % Ben Yakkou and El Fadil \cite{BF} studied the monogenity of pure number fields defined by $x^{p^r}-m$ where $m$ is a square-free integer. 
In \cite{BF, BFC, BDB},  Ben Yakkou \emph{et al.}  investigated  several classes of pure number fields with variable exponents. 
 
In the present paper, we study the monogenity of pure number fields $K = \mathbb{Q}(\sqrt[n]{m})$ without assuming that $m$ is square-free. 

\section{Main results}

Let $p$ be a prime. Throughout this paper, $\F_p$ denotes the finite field with $p$ elements. For $m \in \Z$, $\n_p(m)$ stands for the $p$-adic valuation of $m$, and $PD(m)$ denotes the set of prime divisors of   $m$. % and let $t_p=\frac{t}{p^{\nu_p(t)}}$. 
For two positive integers $d$ and $u$, we denote by $N_p(d)$ the number of monic irreducible polynomials of degree $d$ in $\mathbb{F}_p[x]$, which is given by
\[
N_p(d) = \frac{1}{d} \sum_{t \mid d} \mu(t)\, p^{\tfrac{d}{t}},
\]
where $\mu$ is the Möbius function (see Proposition 4.35 in \cite{Na}). We also denote by $N_p(d,u,m)$ the number of monic irreducible factors of degree $d$ of the polynomial $x^u - m$ in $\mathbb{F}_p[x]$.

Now, keeping the above notation, we present our results. We begin with the following theorem, which provides infinite families of non-monogenic pure number fields; that is, $\mathbb{Z}_K$ has no power integral basis.

\begin{theorem}\label{thgeneral}
Let $F(x) = x^n - m \in \mathbb{Z}[x]$ be an irreducible polynomial, and let $K = \mathbb{Q}(\alpha)$ be a  pure number field generated by a root $\alpha$ of $F(x)$. Assume that  $n=u\cdot p^r$, where  $p$ is an odd prime not dividing $u \cdot m$, and let  $\n=\n_p(m^{p-1}-1)$. If $\min\{r+1, \n\} N_p(d, u, m)> N_p(d)$ for some positive integer $d$, then $K$ is non-monogenic.
\end{theorem}
\begin{remark}
Note that the condition that $m$ is square-free is not required in the above theorem.
\end{remark}
\begin{remark}
The above theorem implies \cite[Theorem 2.4]{BF} and   \cite[Theorem 2.2]{BFC},  
 where the special cases $n= p^r$ and $n=2^r \cdot 5^s$, with  $m$  square-free,  were  previously  studied, respectively.    
\end{remark}
As a consequence of the above result, the following four corollaries provide infinite parametric families of non-monogenic pure number fields with variable exponents.
\begin{corollary}\label{Cor5.7}
For $F(x)=x^{5^r \cdot 7^s}-m$, if any   of the following conditions holds:
\begin{enumerate}
%\item $r \ge 2, s \ge 1$, and $m^6 \equiv 1 \md{49}$.
\item $r \ge 1, s \ge 7$ and $m^6 \equiv 1 \md{7^8}$.
\item  $r \ge 5, s \ge 1 $ and $m^4 \equiv 1 \md{5^6}$,
\end{enumerate}
then $K$ is non-monogenic.
\end{corollary}
\begin{corollary} \label{Cor3.5}  For $F(x)=x^{3^r \cdot 5^s}-m$,  if any of the following conditions holds:
	\begin{enumerate}
		\item  $r \ge 3 $ and $m\equiv \pm 1\md {81}$, 
		\item $s\ge 5$ and  $ \ol{m}  \in \{ \ol{1}, \ol{7}, \ol{18}, \ol{24}\}  \md {25} $,
	\end{enumerate}
	then $K$ is non-monogenic.
\end{corollary}
\begin{corollary}\label{Cor3.11}
	For $F(x)=x^{3^r\cdot 11^s}-m$. If any   of the following conditions holds: 
	\begin{enumerate}
		\item $r \ge 1, s\ge 11$ and $m^{10} \equiv 1 \md{11^{12}}$.
		\item $r\ge 2, s \ge 1$ and $m^2 \equiv 1 \md{27}$,
	\end{enumerate}
	then $K$ is non-monogenic.
\end{corollary}
\begin{corollary}\label{Cor5.11}
For $F(x)=x^{5^r\cdot 11^s}-m$. If any   one of the following conditions holds:
\begin{enumerate}
\item $r \ge 1, s \ge 2, m \equiv -1 \md{11}$ and $m^{10} \equiv 1 \md{1331}$.
\item $r \ge 6, s\ge 1$ and $m^4 \equiv 1 \md{5^6}$,
\end{enumerate}
then $K$ is non-monogenic.
\end{corollary}
Recall that a monic polynomial $F(x) \in \mathbb{Z}[x]$ is called monogenic if it is irreducible over $\mathbb{Q}$ and $\mathbb{Z}_K = \mathbb{Z}[\alpha]$, where $K = \mathbb{Q}(\alpha)$ with $F(\alpha) = 0$. However, it may happen that the field $K$ is monogenic even though its defining polynomial $F(x)$ is not (see, e.g., \cite{JonesASM} by Jones). In the following result, we identify families of monogenic pure number fields defined by non-monogenic binomials.
\begin{theorem}\label{monogen}Let $F(x) = x^n - a^u \in \mathbb{Z}[x]$ be a binomial such that $u \ge 2$, $\gcd(u, n) = 1$, $a$ is square-free, and $PD(n) \subseteq PD(a)$. Then $F(x)$ is irreducible over $\mathbb{Q}$.  Let $K := \mathbb{Q}(\alpha)$ be a pure number field generated by a root $\alpha$ of $F(x)$. Then we have the following:	
	\begin{enumerate}
		\item $\mathbb{Z}[\alpha] \neq \mathbb{Z}_K$ (that is, $F(x)$ is non-monogenic).
		\item There exists $\theta \in \widehat{\mathbb{Z}}_K$ such that $\mathbb{Z}_K = \mathbb{Z}[\theta]$; that is, $K$ is monogenic.
	\end{enumerate}
	
\end{theorem}
\begin{remark}\
\begin{enumerate}
	\item Note that in the above theorem, $a^u$ is not square-free.
	\item Recall that in \cite{Gassert}, the case $\mathbb{Z}[\alpha] = \mathbb{Z}_K$ was studied when $u = 1$. In the above theorem, we have $u \ge 2$; here, $\alpha$ is not a GPIB of $\mathbb{Z}_K$, but $\mathbb{Z}_K$ admits a GPIB $\theta$ different from $\alpha$.
\end{enumerate}

\end{remark}

\section{Preliminaries} 
In what follows, let $K=\mathbb{Q}(\alpha)$ be a number field with ring of integers $\mathbb{Z}_K$, where $\alpha \in \mathbb{Z}_K$ is a root of a monic irreducible polynomial $F(x) \in \mathbb{Z}[x]$ of degree $n$. Recall that for any $\theta \in \widehat{\mathbb{Z}}_K$, $\mathbb{Z}[\theta]$ is a $\mathbb{Z}$-submodule of $\mathbb{Z}_K$ of rank $n$. Hence, the index of $\mathbb{Z}[\theta]$ in $\mathbb{Z}_K$ is finite, denoted by $(\mathbb{Z}_K : \mathbb{Z}[\theta])$. It follows that the field $K$ is  monogenic if $(\mathbb{Z}_K : \mathbb{Z}[\theta]) = 1$ for some primitive element $\theta$ of $K$.

The greatest common divisor of the indices of all integral primitive elements of $K$ is called the  index of $K$, denoted by $i(K)$. That is,  
\[
i(K) = \gcd \{\, (\mathbb{Z}_K : \mathbb{Z}[\theta]) \mid \theta \in \widehat{\mathbb{Z}}_K \,\}.
\]
A prime $p$ is called a common index divisor of $K$ if it divides $i(K)$. Since $i(K)=1$ for every  monogenic number field, the existence  of a common index divisor implies that $K$ is non-monogenic.

For a positive integer $d$, let $L_p(d)$ denote the number of distinct prime ideals of $\mathbb{Z}_K$ that divide $p\mathbb{Z}_K$ and have residue degree $d$. Recall also  the number  $N_p(d)$ defined in section 2. The following lemma provides a sufficient condition for a prime $p$ to be a common index divisor of a number field.

\begin{lemma}\label{lemma}$($ \cite[Theorems 4.33 and 4.34]{Na}$)$ With the above notation,  if $L_p(d) > N_p(d)$ for some positive integer $d$, then $p$ is a  common index divisor of $K$. In particular, $K$ is non-monogenic.
\end{lemma}

To prove our result, we shall use Ore's index theorem, which provides the value of $\nu_p((\mathbb{Z}_K:\mathbb{Z}[\theta]))$ and determines  the factorization of $p\mathbb{Z}_K$. This theorem allows us both to apply the above lemma (to detect non-monogenity) and to check whether $(\mathbb{Z}_K:\mathbb{Z}[\theta])= 1$  (for monogenity). Ore's theorem relies on some Newton polygon techniques. For further details, see  \cite{O} by Ore, \cite{MN92} by Montes and Nart, and \cite{Narprime, Nar} by  Gu\`{a}rdia \emph{et al.}

Following the presentation in \cite{Narprime,Nar} by  Gu\`{a}rdia,  Montes, and  Nart, we briefly summarize the key concepts and outline the application of this method. 
 
 Let $p$ be a prime, and let $\nu_p$ denote the $p$-adic valuation on  $\mathbb{Q}_p(x)$. For any polynomial $F(x) = \sum_{i=0}^n a_i x^i \in \mathbb{Z}_p[x]$, the $p$-adic valuation of $F(x)$ is defined by
 \[
 \nu_p(F(x)) = \min \{\nu_p(a_i) \mid 0 \le i \le n\}.
 \]
Let $\phi(x) \in \mathbb{Z}[x]$ be a monic polynomial whose reduction modulo $p$, $\overline{\phi(x)}$, is irreducible in $\mathbb{F}_p[x]$. 
 Then, the $\phi$-adic development of $F(x)$ is given as follows:  
 \[
 F(x) = a_l(x)\phi(x)^l + a_{l-1}(x)\phi(x)^{l-1} + \cdots + a_1(x)\phi(x) + a_0(x),
 \]
 where $\deg(a_j(x)) < \deg(\phi(x))$ for each $j = 0, 1, \dots, l$.
We consider the following set of points in the euclidean plane: 
\[
\mathfrak{S}_{\{F, \phi\}} = \{ (j, \nu_p(a_j(x))) \mid a_j(x) \neq 0 \text{ and } 0 \le j \le l \}.
\]
The lower convex hull of $\mathfrak{S}_{\{F, \phi\}}$ is called the    $\phi$-Newton polygon of $F(x)$, denoted by $N_{\phi}(F)$. If it  consists of $g$  distinct sides $S_k$ with distinct slopes $\lambda_k \in \mathbb{Q}$, then we write
\[
N_{\phi}(F) = S_1 + S_2 + \cdots + S_g.
\]
 In particular, the $\phi$-principal Newton polygon of $F(x)$, denoted by $N_{\phi}^+(F)$, is the polygon formed by the sides of $N_{\phi}(F)$ that  have negative slopes. 
 
Every side $S$ of $N_{\phi}^+(F)$, joining the points $(s, \nu_p(a_s(x)))$ and $(r, \nu_p(a_r(x)))$, has length $l(S) := r - s$ and height $h(S) := \nu_p(a_s(x)) - \nu_p(a_r(x))$. Its slope is $\lambda_{S} := -\frac{h(S)}{l(S)} < 0$, its degree is $d(S) := \gcd(l(S), h(S))$, and its ramification index is $e(S) := \frac{l(S)}{d(S)}$. Let $d = d(S)$ and $e = e(S)$, and define $\mathbb{F}_{\phi} := \mathbb{F}_p[x]/(\overline{\phi(x)})$. Then, the residual polynomial of $F(x)$ associated to $S$ (or to $\lambda_S$) is
\[
R_{\lambda_S}(F)(y) = c_{s+de}y^d + c_{s+(d-1)e}y^{d-1} + \cdots + c_{s+e}y + c_s \in \mathbb{F}_{\phi}[y],
\]
 where \[
 c_{s+ie} = \frac{a_{s+ie}(x)}{p^{\nu_p(a_{s+ie}(x))}} \bmod (p, \phi(x)), \quad \text{for all } i = 0, 1, \dots, d.
 \] 
The  polynomial $F(x)$ is $\phi$-regular if each residual polynomial $R_{\lambda_{S}}(F)(y)$ corresponding to a side $S_{\lambda}$ of $N_{\phi}^+(F)$ is separable in $\mathbb{F}_{\phi}[y]$. Subsequently, $F(x)$ is  $p$-regular if it is $\phi$-regular for every monic polynomial  $\phi(x) \in \mathbb{Z}[x]$ whose reduction modulo  $p$ is irreducible and  divides $\overline{F(x)}$ in $\mathbb{F}_p[x]$.

We now have all the necessary tools to state Ore's theorem (see \cite{Narprime, Nar, MN92, O}).
\begin{theorem}[Ore's Theorem]\label{Ore}
With the above notations, suppose that:
	\begin{enumerate}[label=\roman*)]
		\item 
	$
		\overline{F(x)} = \prod_{i=1}^t \overline{\phi_i(x)}^{\,l_i},
$
		where the $\overline{\phi_i(x)}$ are distinct monic irreducible polynomials in $\mathbb{F}_p[x]$.
		
		\item For each $i=1, \ldots, t$, 	$
		N_{\phi_i}^+(F) = S_{i1} + S_{i2} + \dots + S_{ir_i}
		$.
		
		\item For each $i=1, \ldots, t$ and $j=1, \ldots, r_i$, 
	$
		R_{\lambda_{ij}}(F)(y) = \prod_{s=1}^{s_{ij}} \psi_{ijs}(y)^{\,n_{ijs}},
	$
	where $\psi_{ijs}(y)$	are distinct irreducible polynomials in  $\mathbb{F}_{\phi_i}[y]$.
	\end{enumerate}
	
	Then, the following hold:
	\begin{enumerate}
		\item 
		$
		\nu_p((\mathbb{Z}_K : \mathbb{Z}[\alpha])) \ge \sum_{i=1}^t \mathrm{ind}_{\phi_i}(F),
	$
		with equality if $F(x)$ is $p$-regular (i.e., $n_{ijs} = 1$ for all $i,j,s$).
		
		\item If $F(x)$ is $p$-regular,  then
		\[
		p \mathbb{Z}_K = \prod_{i=1}^t \prod_{j=1}^{r_i} \prod_{s=1}^{s_{ij}} \mathfrak{p}_{ijs}^{\,e_{ij}},
		\]
		where $e_{ij}$ is the ramification index of the side $S_{ij}$, $\mathfrak{p}_{ijs}$ is a prime ideal of $\mathbb{Z}_K$ whose residue degree over $p$ is 
$
		f_{ijs} = \deg(\phi_i) \cdot \deg(\psi_{ijs}).
	$
	\end{enumerate}
\end{theorem}

\begin{example}\label{Example}
	
Consider the quartic pure field $K=\Q(\sqrt[4]{17})$. In this case, we have  $\a=\sqrt[4]{17}$ and $F(x)=x^4-17$. Let $p=2$. Reducing modulo $2$, we get  $F(x)=\ph_1(x)$ in $\F_2[x]$, where $\ph_1(x)=x-1$. The $\ph_1$-adic development of $F(x)$ is 
\[
F(x)=-16+4\phi_1(x)+6\phi_1(x)^2+4\phi_1(x)^3+\phi_1(x)^4.
\]
Here, we have \[
\nu_0=\nu_2(-16)=4, \quad 
\nu_1=\nu_2(4)=2, \quad 
\nu_2=\nu_2(6)=1, \quad 
\nu_3=\nu_2(4)=2, \,\,
\text{and} \,\, \nu_4=\nu_2(1)=0.
\] It follows that $N_{\ph_1}^{+}(F)=S_{11}+S_{12}+S_{13}$ has three distinct  sides of degree $1$ each,  joining the points $(0, 4), (1, 2), (2, 1)$ and $(4, 0)$ (see FIGURE 1). Moreover, for $j=1,2,3$, the residual polynomial $
R_{\lambda_{1j}}(F)(y)=y+1
$
is of degree $1$. Hence, it is separable over $\mathbb{F}_{\phi_1}\simeq \mathbb{F}_2$. Therefore, $F(x)$ is $2$-regular.  By Theorem~\ref{Ore}, we have  
\[
2\mathbb{Z}_K = \mathfrak{p}_{111} \cdot \mathfrak{p}_{121} \cdot \mathfrak{p}_{131}^2 
\]
where $\mathfrak{p}_{1j1} $  are distinct prime ideals of $\Z_K$ with residue degree $f(\mathfrak{p}_{1j1}/2)=1$ for $j=1,2,3$.
Moreover,   we have 
\[
\nu_2\big((\mathbb{Z}_K : \mathbb{Z}[\alpha])\big) = \operatorname{ind}_{\phi_1}(F) = 3.
\]
In this case, by Lemma \ref{lemma}, the prime $2$ divides $i(K)$. Hence, $K$ is non-monogenic.

\end{example}

 \begin{figure}[htbp]

	\centering
	
	\begin{tikzpicture}[x=2cm,y=0.35cm]
		\draw[latex-latex] (0,5) -- (0,0) -- (4.25,0) ;

		\draw[thick] (0,0) -- (-0.2,0);
		\draw[thick] (0,0) -- (0,-0.5);
		
		\draw[thick,red] (1,-2pt) -- (1,2pt);
		\draw[thick,red] (2,-2pt) -- (2,2pt);
			\draw[thick,red] (3,-2pt) -- (3,2pt);
			\draw[thick,red] (4,-2pt) -- (4,2pt);
		%	\draw[thick,red] (5,-2pt) -- (5,2pt);
		%	\draw[thick,red] (6,-2pt) -- (6,2pt);
		%	\draw[thick,red] (7,-2pt) -- (7,2pt);
		%	\draw[thick,red] (8,-2pt) -- (8,2pt);
		%\draw[thick,red] (9,-2pt) -- (9,2pt);
		\draw[thick,red] (-2pt,1) -- (2pt,1);
		\draw[thick,red] (-2pt,2) -- (2pt,2);
			\draw[thick,red] (-2pt,3) -- (2pt,3);
			\draw[thick,red] (-2pt,4) -- (2pt,4);	
		%\draw[thick,red] (-2pt,5) -- (2pt,5);	
		\node at (1,0) [below ,blue]{\footnotesize  $1$};
		\node at (2,0) [below ,blue]{\footnotesize $2$};
			\node at (3,0) [below ,blue]{\footnotesize  $3$};
			\node at (4,0) [below ,blue]{\footnotesize  $4$};
		%	\node at (5,0) [below ,blue]{\footnotesize  $5$};
		%	\node at (6,0) [below ,blue]{\footnotesize  $6$};
		%	\node at (7,0) [below ,blue]{\footnotesize $7$};
		%	\node at (8,0) [below ,blue]{\footnotesize  $8$};
		\node at (0,1) [left ,blue]{\footnotesize  $1$};
		\node at (0,2) [left ,blue]{\footnotesize  $2$};
		\node at (0,3) [left ,blue]{\footnotesize  $3$};
		\node at (0,4) [left ,blue]{\footnotesize  $4$};
		%\node at (0,5) [left ,blue]{\footnotesize  $5$};
		\draw[thick, mark = *] plot coordinates{(0,4) (1,2) (2,1) (4,0) };
		%\draw[thick, only marks, mark=*] plot coordinates{ (2,1) (3,1)  };
		%\draw[thick, only marks, mark=*] plot coordinates{(3,3) (1,3) (2,2) (5,3) (7,3) (6,2) };	
		\node at (0.5,2.8) [above  ,blue]{\footnotesize $S_{1}$};
		\node at (1.5,1.2) [above   ,blue]{\footnotesize $S_{2}$};
		\node at (3,0.3) [above   ,blue]{\footnotesize $S_{3}$};
	\end{tikzpicture}
	\caption{ \small $N_{\ph_1}^+(F)$ with respect to $\n_2$.}
\end{figure}
\section{Proofs of the main results}
In this section, we prove our results. To prove  Theorem~\ref{thgeneral}, we will use Lemma~\ref{Polygon} below. Let $F(x) = x^n - m$, and let $p$ be a prime. Assume that $p$ divides $n$ but does not divide $m$, and write $n = u \cdot p^r$. Let  $\ph(x) \in \Z[x]$ be a monic irreducible polynomial.  Then 
$\ol{\ph(x)}$  is an irreducible factor of $\ol{F(x)}$ in $\F_p[x]$ if and only if  $\overline{\ph(x)}$ is an  irreducible factor of  $\overline{x^{u}-m}$ in $\F_p[x]$. Let 
\begin{itemize}
	\item $x^{u}-m = \phi(x)\,U(x) + p\,T(x)$, where $U(x)$ and $T(x)$ are two polynomials in $\mathbb{Z}[x]$ such that $\overline{\phi(x)}$ does not divide $\overline{U(x)T(x)}$.
	\item $H(x) = m^{p^r-1}T(x) + \frac{1}{p^{r+1}} \sum_{j=0}^{p^r-2} \binom{p^r}{j} m^j (p T(x))^{p^r-j}$.
	\item Let $V(x)$ and $R(x)$ be the quotient and the remainder, respectively, upon the Euclidean division of $H(x)$ by $\phi(x)$.
	\item $A_0(x) = p^{r+1} R(x) + m^{p^r} - m$, and $\nu_0 = \nu_p(A_0(x))$.
\end{itemize}
In \cite{BDB}, Ben Yakkou and Didi  proved the following result:
\begin{lemma}\label{Polygon}
	Let $F(x)$, $p$, $\nu_0$, and $\phi(x)$ be as above. Then $N_{\phi}^{+}(F)$ is the Newton polygon joining the points 
	\[
	\{(0, \nu_0)\} \cup \{(p^j, r-j) \mid 0 \le j \le r\}
	\]
	in the euclidean plane.
\end{lemma}	
We  recall also that  if a prime   $p$ does not divide a non-zero rational integer $m$, then for every positive integer $r$,  $\n=  \nu_p (m^{p^r}- m )=\nu_p (m^{p^r - 1}- 1 ) = \nu_p{ (m^{p -1} - 1 )}$ (e.g., \cite{BF} or Remark on p. 8 of \cite{BFC} by Ben Yakkou \emph{et al.}).

\begin{proof}[\textbf{Proof of Theorem \ref{thgeneral}}]
Under the assumptions of Theorem \ref{thgeneral},  let $\ph(x)$ be a monic irreducible factor of $F(x)$ modulo $p$ of degree $d$. We distinguish two cases:
 \begin{itemize}
     \item If $\nu \ge r+1$, then $\nu_0 \ge \min\{\nu, r+1\} \ge r+1$. It follows from Lemma~\ref{Polygon} that the principal Newton polygon 
     \[
     N_{\phi}^{+}(F) = S_1 + \cdots + S_{r+1}
     \] 
     has $r+1$ distinct sides, each of degree $1$. More precisely, $N_{\phi}^{+}(F)$ is the lower convex hull of the points $(0, \nu_0)$, $(1, r)$, $(p, r-1)$, $\ldots$, $(p^r, 0)$ (see FIGURE~2).   Thus, for $j = 1, \ldots, r+1$, the residual polynomial $R_{\lambda_j}(F)(y)$ is separable in $\mathbb{F}_{\phi}[y]$, since it is of  degree $1$. By Ore's Theorem (Theorem~\ref{Ore}), the factor $\phi(x)$ provides  $r+1$ distinct prime ideals of $\mathbb{Z}_K$, each of residue degree 
$
     \deg(R_{\lambda_j}(F)(y)) \times \deg(\phi(x)) = 1 \times d = d,
   $
     lying above the prime $p$. Since the number of monic irreducible factors of $\overline{F(x)}$ is $N_p(d,u,m)$, there exist $(r+1) \times N_p(d,u,m)$ distinct prime ideals of $\mathbb{Z}_K$ dividing $p\mathbb{Z}_K$.
     
     \item If $r \ge \nu$, then $\nu_0 \ge \min\{\nu, r+1\} = \nu$. Thus, by Lemma~\ref{Polygon}, the Newton polygon 
     \[
     N_{\phi}^{+}(F) = S_1 + \cdots + S_{\nu}
     \] 
     has $\nu$ distinct sides, each of degree $1$. Explicitly, $N_{\phi}^{+}(F)$ is the lower convex hull of the points 
     \[
     (0, \nu), \; (p^{\,r-\nu+1}, \nu-1), \; (p^{\,r-\nu+2}, \nu-2), \; \ldots, \; (p^r, 0).
     \]     Hence, $R_{\lambda_j}(F)(y)$, for $j = 1, \ldots, \nu$, is separable. Therefore, for $k = 1, \ldots, N_p(d, u, m)$, the polynomial $F(x)$ is $\phi_k$-regular for every factor $\phi_k(x)$ of $F(x)$ modulo $p$ of degree $d$. Moreover,   the ramification indices of the sides of $N_{\phi}^{+}(F)$ are 
     \[
     e_1 = e(S_1) = p^{\,r-\nu+1}, \quad 
     e_j = e(S_j) = (p-1)\,p^{\,r-\nu+j-1}, \quad j = 2, \ldots, \nu.
     \]    Applying  Theorem~\ref{Ore}, we obtain
     \[
     p \, \mathbb{Z}_K = \prod_{k=1}^{N_p(d,u,m)} \prod_{j=1}^{\nu} \mathfrak{p}_{kj}^{\,e_j} \, \mathfrak{a},
     \]
     where $\mathfrak{a}$ is a proper ideal of $\mathbb{Z}_K$, and $\mathfrak{p}_{kj}$ is a prime ideal of $\mathbb{Z}_K$ with residue degree 
   \[
   f(\mathfrak{p}_{kj} / p) = d, \quad 
   k = 1, \ldots, N_p(d,u,m), \quad 
   j = 1, \ldots, \nu.
   \]
    Hence, there are at least $\nu \times N_p(d,u,m)$ distinct  prime ideals of residue degree $d$ lying above the prime  $p$.
     
     We conclude from Lemma~\ref{lemma} that  if 
   $
     \min\{\nu, r+1\} \times N_p(d,u,m) > N_p(d),
  $
     then $p$ divides $i(K)$, and therefore $K$ is non-monogenic.

 \end{itemize}
\end{proof}	
\begin{figure}[htbp] 
	\centering
	\begin{tikzpicture}[x=0.4cm,y=0.5cm]
		\draw[latex-latex] (0,6.5) -- (0,0) -- (29,0) ;
		\draw[thick] (0,0) -- (-0.5,0);
		\draw[thick] (0,0) -- (0,-0.5); 
		%	\draw[thick,red] (1,-2pt) -- (1,2pt);
		\draw[thick,red] (3,-2pt) -- (3,2pt);
		\draw[thick,red] (6,-2pt) -- (6,2pt);
		\draw[thick,red] (11,-2pt) -- (11,2pt);
		\draw[thick,red] (18,-2pt) -- (18,2pt);
		\draw[thick,red] (28,-2pt) -- (28,2pt);
		\draw[thick,red] (-2pt,1) -- (2pt,1);
		\draw[thick,red] (-2pt,2) -- (2pt,2);
		%	\draw[thick,red] (-2pt,3) -- (2pt,3);
		\draw[thick,red] (-2pt,3.7) -- (2pt,3.7);
		\draw[thick,red] (-2pt,5.5) -- (2pt,5.5);
		%\draw[thick,red] (-2pt,5.5) -- (2pt,5.5);
		%	\draw[thick,red] (-2pt,6) -- (2pt,6);
		\node at (0,0) [below left,blue]{\footnotesize  $0$};
		%\node at (1,0) [below ,blue]{\footnotesize  $1$};
		\node at (3,0) [below ,blue]{\footnotesize $1$};
		\node at (6,0) [below ,blue]{\footnotesize  $p$};
		\node at (11,0) [below ,blue]{\footnotesize  $p^{r-2}$};
		\node at (18,0) [below ,blue]{\footnotesize  $p^{r-1}$};
		\node at (28,0) [below ,blue]{\footnotesize  $p^{r}$};
		\node at (0,1) [left ,blue]{\footnotesize  $1$};
		\node at (0,2) [left ,blue]{\footnotesize  $2$};
		%	\node at (0,3) [left ,blue]{\footnotesize  $3$};
		\node at (0,3.7) [left ,blue]{\footnotesize  $r$};
		%	\node at (0,5) [left ,blue]{\footnotesize  $\mu$};
		\node at (0,5.5) [left ,blue]{\footnotesize  $\n_0 \ge \n \ge r+1$};
		%	\node at (0,5.5) [left ,blue]{\footnotesize  $r+1$};
		%	\node at (0,6) [left ,blue]{\footnotesize  $\delta \ge 6$};
		\draw[thick,mark=*] plot coordinates{(0,5.5) (3,3.7)};
		\draw[thick,mark=*] plot coordinates{(3,3.7) (6,3)};
		%\draw[thick,mark=*] plot coordinates{(6,3) (10,2)};
		\draw[thick,mark=*] plot coordinates{(11,2) (18,1)};
		\draw[thick,mark=*] plot coordinates{(28,0) (18,1)};
		%	\draw[thick, dashed] plot coordinates{(1.4,4.8) (2.6,4.1) };
		\draw[thick, dashed] plot coordinates{(6.5,2.9) (10.5,2.1) };
		\node at (0.5,4.7) [above right  ,blue]{\footnotesize  $S_{1}$};
		\node at (4.5,3.2) [above right  ,blue]{\footnotesize  $S_{2}$};
		\node at (14,1.4) [above right  ,blue]{\footnotesize  $S_{r}$};
		\node at (23,0.3) [above right  ,blue]{\footnotesize  $S_{r+1}$};
		%\node at (50,0.3) [above right  ,blue]{\footnotesize  $S_{5}$};
	\end{tikzpicture}
	\caption{    \small  $\npp{F}$  when  $\nu \ge r+1$.\hspace{5cm}}
\end{figure}
The proofs of Corollaries~\ref{Cor5.7}, \ref{Cor3.5}, \ref{Cor3.11}, and \ref{Cor5.11} are similar. Since they follow directly from Theorem~\ref{thgeneral}, we suggest presenting only the proof   of Corollary~\ref{Cor5.7}.

\newpage
\begin{proof}[\textbf{Proof of Corollary  \ref{Cor5.7}}]\
	\begin{enumerate}
	\item  Assume that $r \ge 1$, $s \ge 7$, and $m^6 \equiv 1 \pmod{7^8}$. Here, we have
	\[
	F(x) = x^{5^r \cdot 7^s} - m \equiv (x^{5^r} - m)^{7^s} \equiv (\phi_1(x) V(x))^{7^s}\pmod{7},
	\]
	with $\phi_1(x) = x - b_m$ and $V(b_m) \not\equiv 0 \pmod{7}$ (since $x^{5^r} - m$ is square-free in $\mathbb{F}_7[x]$, because $7$ does not divide $5^r$). 	Precisely, modulo the prime $7$, we have $b_1 = 1$, $b_2 = 4$, $b_3 = 5$, $b_4 = 2$, $b_5 = 3$, and  $b_6 = 6$. Hence, $\phi_1(x)$ is an irreducible factor of $x^{5^r} - m$ modulo $7$. According to the notations of Theorem~\ref{thgeneral}, for $d = 1$, we have $N_7(1, 5^r, m) = 1$. Notice also that $N_7(1) = 7$.	By applying Theorem~\ref{thgeneral}, $K$ is non-monogenic if the condition
$
	\min\{ s+1, \nu_7(m^6 - 1) \} \times 1 > 7
$
	holds, which is equivalent to $s \ge 7$ and $m^6 \equiv 1 \pmod{7^8}$.

	\item  Suppose that $r \ge 5$, $s \ge 1$, and $m^4 \equiv 1 \pmod{5^6}$. In this case, for $p = 5$ and $d = 1$, we have $N_5(1, 7^s, m) = 1$ and $N_5(1) = 5$. Then, similarly to the previous case, by interchanging the roles of $r$ and $s$ and applying Theorem~\ref{thgeneral}, we obtain the desired result.
	
	\end{enumerate}

\end{proof}
Now we proceed to the proof of Theorem \ref{monogen}.
\begin{proof}[\textbf{Proof of Theorem \ref{monogen}}]
Let $p$ be a prime dividing $a$. Then $F(x) \equiv \phi_1(x)^n \pmod{p}$, where $\phi_1(x) = x$. Note that $\nu_p(a^u) = u$, since $a$ is square-free. As $\gcd(u, n) = 1$, the Newton polygon $N_{\phi_1}^{+}(F) = S_1$ has only one side of degree $1$, joining the points $(0, u)$ and $(n, 0)$. Hence, $\deg(S_1) = \gcd(u, n) = 1$, so the degree of the residual polynomial $R_{\lambda_1}(F)(y)$ is $1$. By the theorem of the residual polynomials (see \cite{Nar}), $F(x)$ is irreducible over $\mathbb{Q}_p$, and therefore it is irreducible over $\mathbb{Q}$.

	\begin{enumerate}
		\item Since $\deg(R_{\lambda_1}(F)(y)) = 1$, the polynomial $R_{\lambda_1}(F)(y)$ is separable over $\mathbb{F}_{\phi_1}[y]$. Hence, $F(x)$ is $\phi_1$-regular, and consequently, it is $p$-regular. By  Theorem~\ref{Ore}(1) and Pick's Theorem,  we have
		\[
		\nu_p(\mathbb{Z}_K : \mathbb{Z}[\alpha]) \ge \operatorname{ind}_{\phi_1}(F) = \frac{(n-1)(u-1)}{2} \ge 2.
		\]
		It follows that $\mathbb{Z}[\alpha] \subsetneq \mathbb{Z}_K$.
		\item  Let $(t, s) \in (\mathbb{N}^*)^2$ be a positive solution of the Diophantine equation $u x - n y = 1$, and set 
	$
		\theta = \frac{\alpha^t}{a^s}.
	$
		Then we have
		\[
		\theta^n = \left(\frac{\alpha^t}{a^s}\right)^n = \frac{\alpha^{tn}}{a^{sn}} = \frac{(\alpha^n)^t}{a^{sn}} = \frac{a^{ut}}{a^{sn}} = a^{ut - ns} = a.
		\]Thus, $G(\theta) = 0$, where $G(x) = x^n - a \in \mathbb{Z}[x]$. Since $a$ is square-free, $F(x)$ is irreducible over $\mathbb{Q}$ (because it is a $q$-Eisenstein-type polynomial for any prime $q$ dividing $a$).\\
		\\ Consider the field $K' = \mathbb{Q}(\theta)$. Clearly, $K' \subseteq K$. On the other hand, $\alpha^t = a^s \theta \in K'$. It follows that
	\[
	\alpha = \alpha^1 = \alpha^{ut - ns} = \frac{\alpha^{ut}}{\alpha^{ns}} = \frac{(\alpha^t)^u}{(\alpha^n)^s} = \frac{(\alpha^t)^u}{a^s} \in K'.
	\]	Therefore, $K=\Q(\a)=\Q(\th)$. By \cite[Propositions 2.9 and 2.13]{Na}, we have
	\begin{equation*}\label{discriminant}
		\Delta(G) = n^n \cdot a^{n-1} = (\mathbb{Z}_K : \mathbb{Z}[\theta])^2 \cdot D_K,
	\end{equation*}
	where $D_K$ is the field discriminant of $K$.\\
	\\
	Let $q$ be a prime. Since $PD(n) \subseteq PD(a)$, it follows from the above relation that if a prime $q$ divides the index $(\mathbb{Z}_K : \mathbb{Z}[\theta])$, then $q$ divides $a$. Therefore, to show that $\mathbb{Z}_K = \mathbb{Z}[\theta]$, it suffices to prove that $q \nmid (\mathbb{Z}_K : \mathbb{Z}[\theta])$ for every $q \in PD(a)$.\newline
	\\ Let $q \in PD(a)$. We have $G(x) \equiv \phi_2(x)^n \pmod{q}$, where $\phi_2(x) = x$. Since $a$ is square-free, $\nu_q(a) = 1$. It follows that the $\phi_2$-Newton polygon of $G$ with respect to $q$, $N_{\phi_2}^{+}(G) = T_1$,  has only one side of degree 1, joining the points $(0, 1)$ and $(n, 0)$. By applying Ore's Theorem to $G(x)$, we conclude that $\nu_q((\mathbb{Z}_K : \mathbb{Z}[\theta])) = 0$. This completes the proof of the theorem.

	\end{enumerate}
\end{proof}
\section{Application to the existence of Canonical number systems}
Let $N_{K/\Q}$ denote the norm function over $K$. The pair $(\theta, N_{\theta})$, where $\theta \in \Z_K$ and 
\[
N_{\theta} = \{0,1,\ldots, |N_{K/\Q}(\theta)|-1\},
\]
is called a canonical number system (CNS) in $\mathbb{Z}_K$ if every element of $\mathbb{Z}_K$ can be uniquely represented in the form
$
a_0 + a_1 \theta + \cdots + a_l \theta^l,
$
with coefficients $a_i \in N_{\theta}$ for all $i=0,\ldots,l$, and $a_l \neq 0$.
  We call $\theta$ the base of the number system, $N_{\theta}$ the set of digits, and the minimal polynomial $F(x)$ of $\theta$ the CNS polynomial. For further details and applications, we refer to \cite{EGP} by Evertse \emph{et al.} and \cite{Prthocrypto} by Peth\H{o}. The most important connection between CNS and monogeneity in number fields is given by the following result (see \cite{Kovacs} by Kov\'acs, and \cite{KovacsPetho} by Kov\'acs and Peth\H{o}).
  
\begin{proposition} \label{Kov}
Let $K$ be a number field of degree $n \geq 3$. Suppose $\theta \in \mathbb{Z}_{K}$ has minimal polynomial
$
F(x) = x^{n} + a_{n-1}x^{n-1} + \cdots + a_{0},
$
with coefficients satisfying
\[
1 \leq a_{n-1} \leq \cdots \leq a_{0}, \, a_{0} \geq 2, \, \text{and} \, |N_{K/\mathbb{Q}}(\theta)| > 2.
\]
Then the pair $(\theta, N_{\theta})$ forms a canonical number system in $\mathbb{Z}_{K}$ if and only if
$
\mathbb{Z}_{K} = \mathbb{Z}[\theta].
$
\end{proposition}
Based on the above proposition, Hameed, Nakahara, and Husnine identified in Theorems~2.1 and~3.1 of \cite{HNHCNS} the CNS polynomials of the forms $x^{2^r} - m$ and $x^n - m$, under the condition that $m$ is square-free. Combining the above proposition with our results, we then obtain the following:
\begin{itemize}
	\item Under the hypotheses of Theorem~\ref{monogen}, $\mathbb{Z}_{K}$ admits a CNS, and such a system can be explicitly constructed by combining Proposition~\ref{Kov} with the proof of Theorem~\ref{monogen}. In this context, our results extend those given in \cite{HNHCNS}.
	\item Under the hypotheses of Theorem~\ref{thgeneral}, $\mathbb{Z}_{K}$ does not admit any CNS.
\end{itemize}
To illustrate our results, we present the following examples.
\newpage
	\begin{example}
		Let $F(x) = x^6 - m^5$ with $m \neq 0, \pm 1$, $m$ square-free, and divisible by 6, and let $K = \mathbb{Q}(\sqrt[6]{m^5})$ (here $\alpha = \sqrt[6]{m^5}$). According to the notations in the proof of Theorem~\ref{monogen}, we have $t = 5$, $s = 4$, and
	$
		\theta = \frac{(\sqrt[6]{m^5})^5}{m^4}.
$
		By Theorem~\ref{monogen}, $\theta$ generates a power integral basis of $\mathbb{Z}_K$. Note that the minimal polynomial of $\theta$ is $G(x) = x^6 - m$. By \cite[Propositions 2.9 and 2.13]{Na}, we have
		\[
		6^6 \cdot m^5 = \Delta(G) = D_{K/\mathbb{Q}}(1, \theta, \ldots, \theta^5) = (-1)^{\frac{6 \cdot 5}{2}} N_{K/\mathbb{Q}}(G'(\theta)) = -N_{K/\mathbb{Q}}(6 \theta^5) = -6^6 \bigl(N_{K/\mathbb{Q}}(\theta)\bigr)^5.
		\]
		Thus, $|N_{K/\mathbb{Q}}(\theta)| = |m|$, so $N_{\theta} = \{0, 1, \ldots, |m|-1\}$. Hence, by Proposition~\ref{Kov}, $(\theta, N_{\theta})$ is a CNS in $\mathbb{Z}_K$, and $G(x) = x^6 - m$ is a CNS polynomial.
		
		As a particular case, if $m = 30$, then every element of $\mathbb{Z}_K$ can be uniquely expressed as
	$
		\sum_{i=0}^{l} a_i \theta^i,
		$
		with coefficients $a_i \in   \{0,1, \ldots, 29\}
		$ for $i = 0, \ldots, l$, and $a_l \neq 0$.
	\end{example}
\begin{example}
	Let $F(x) = x^{2^r \cdot 3^k} - 6^w$ with $w \ge 5$ and $w \equiv 1, 5 \pmod{6}$. Let $t, s$, and $\theta$ be as in the proof of Theorem~\ref{monogen}. Then $(\theta, N_{\theta})$ forms a CNS in $\mathbb{Z}_K$, where the set of digits is $N_{\theta} = \{0, 1, \ldots, 5\}$.
\end{example}
\section*{Declarations}
\textbf{Conflict of interest:} There is no conflict of interest related to this paper. The authors have freely chosen this journal without any consideration.
\section*{Data availability} 
Not applicable.

%\section{Examples}

\end{document}